\documentclass[10pt]{article}
\usepackage{amsmath, amsfonts, amsthm, amssymb,verbatim,graphicx,color}
        \newtheorem{theorem}{Theorem}[section]
        \newtheorem{proposition}[theorem]{Proposition}
        \newtheorem{definition}[theorem]{Definition}

\newtheorem{claim}{Claim}
\numberwithin{equation}{section}

\usepackage{mathabx,a4wide}

\newcommand \del        \partial
\newcommand{\auth}{\textsc}

\newcommand \be     {\begin{equation}}
\newcommand \ee     {\end{equation}}
\newcommand \RR      {\mathbb{R}}
\newcommand \eps     \epsilon

\let\oldmarginpar\marginpar
\renewcommand\marginpar[1]{\-\oldmarginpar[\raggedleft\footnotesize #1]%
{\raggedright\footnotesize #1}}
\begin{document}

\title{The relativistic Burgers equation on a FLRW background
\\
 and its finite volume approximation}

\author{
%     correct spelling:   LeFloch  or LeFLOCH
Tuba Ceylan$^1$, Philippe G. LeFloch$^2$,  and Baver Okutmustur$^1$}

\date{December 2015}

\maketitle

\footnotetext[1]{Department of Mathematics, Middle East Technical University (METU), 06800 Ankara, Turkey.
 E-mail: {\sl ceylanntuba@gmail.com, baver@metu.edu.tr}}

\footnotetext[2]{Laboratoire Jacques-Louis Lions \& Centre National de la Recherche Scientifique,
Universit\'e Pierre et Marie Curie (Paris VI), 4 Place Jussieu, 75252 Paris, France.
E-mail: {\sl contact@philippelefloch.org.}

}

%====================================================================

\begin{abstract}
A relativistic generalization of the inviscid Burgers equation was proposed by LeFloch, Makhlof, and Okutmustur and then investigated on a Schwarz\-schild background. Here, we extend their analysis to a Friedmann--Le\-ma\^{i}tre--Robertson--Walker (FLRW) background.
This problem is more challenging due to the existence of non-trivial spatially homogeneous solutions.
First, we derive the relativistic Burgers model of interest and determine its spatially homogeneous solutions. Second, we design a numerical scheme based on the finite volume methodology, which is well-balanced in the sense that spatially homogeneous solutions are preserved at the discrete level of approximation.
Numerical experiments demonstrate the efficiency of the proposed method for weak solutions containing shock waves.
\end{abstract}

%
%\setcounter{tocdepth}{5}
%\tableofcontents
% 

%===================================================================

\section{Introduction}

\subsection*{Aim of this paper}

The inviscid Burgers equation is an important model in computational
fluid dynamics, and represents the simplest (yet challenging)
example of a nonlinear hyperbolic conservation law. Recently,
several relativistic and non-relativistic generalizations of the
classical Burgers equation have been introduced by LeFloch and
collaborators \cite{ALO,LO,LO2,LMO}, which also take into account
geometrical effects. In particular, the fundamental {\sl
relativistic Burgers equation} was derived by identifying a
hyperbolic balance law which satisfies the same Lorentz invariance
property as the one satisfied by the Euler equations of relativistic
compressible fluids. The relativistic generalization of this model
was studied on both a flat background and a Schwarzschild
background. A numerical scheme was developed by using the finite
volume methodology and allowed to capture discontinuous solutions
containing shock waves for the relativistic Burgers equation.

The lack of maximum or total variation diminishing principles is lacking for the model under consideration in this work, and the numerical analysis of this model is therefore particularly challenging. Our main objective is designing an accurate and robust numerical approximation method.  

Specifically, we will work here on  Friedmann--Lema\^{i}tre--Robertson--Walker (FLRW)
background, which is an important solution to Einstein's field equations relevant to cosmology. (See for instance \cite{HEL} for background material.)
The main purpose of the article is to discuss the relativistic Burgers equation on a FLRW background  and to design a finite volume scheme for its approximation by closely following LeFloch, Makhlof, and Okutmustur \cite {LMO}.

In the present paper, we continue this analysis and introduce the class of {\sl relativistic Burgers equation on a curved background}, derived as follows. We start from
the relativistic Euler equations on a curved background $(M,g)$ (that is, a smooth, time-oriented Lorentzian manifold), which read
\be
\label{Euler0}
\aligned
 \nabla_\alpha T^{\alpha\beta} &=0,
\\
 T^{\alpha\beta}
&= (\rho c^2 + p) \, u^\alpha u^\beta + p \, g^{\alpha\beta},
\endaligned
\ee
where $T^{\alpha\beta}$ is the so-called energy-momentum tensor for perfect fluids. Here, $\rho \geq 0$ denotes the mass-energy density of the fluid, while
 the future-oriented, unit timelike vector field $u=(u^\alpha)$ represents the velocity of the fluid: $g_{\alpha\beta} \, u^\alpha u^\beta = -1$.

As usual, the model \eqref{Euler0} must be supplemented with  an equation of state for the pressure $p = p(\rho)$. In the present work, we assume that the fluid is pressureless, that is,
$p \equiv 0$, so that the Euler system takes the simpler form
\be
\nabla_\alpha \big( \rho \, u^\alpha u^\beta  \big) = 0.
\ee
Provided $\rho>0$ and $\rho, u$ are sufficiently regular and observing that $g_{\alpha\beta} \nabla_\alpha u^\alpha u^\beta=0$ (that is, $u$ is orthogonal to $\nabla u$, as is easiy checked by differentiating the identity stating that $u$ is unit vector), we arrive at
$$
 \rho \nabla_\alpha u^\alpha u^\beta+ \rho u^\alpha  \nabla_\alpha u^\beta + u^\alpha u^\beta   \nabla_\alpha  \rho=0.
$$
By contracting this equation with the covector $u_\beta$, we get
$$
u^\alpha \nabla_\alpha \rho= - \rho \nabla_\alpha u^\alpha,
$$
which gives us
$$
 \rho u^\beta  \nabla_\alpha u^\alpha +  \rho (  u^\alpha  \nabla_\alpha u^\beta-  u^\beta  \nabla_\alpha   u^\alpha )=0.
$$
Provided $\rho>0$, it thus follows that
\be
\label{Euler4}
 u^\alpha  \nabla_\alpha u^\beta=0,
\ee
which is the {\bf geometric relativistic Burgers equation}, which will be the focus of the present paper.

%-----------------------------------------------------------------------------------------------

\subsection*{Relativistic Burgers equations on a curved background}

We rely here on LeFloch, Makhlof, and Okutmustur \cite {LMO} who treated the Minkowski and Schwarzschild spacetimes. First of all, the standard inviscid Burgers equation is one of the simplest example of nonlinear hyperbolic conservation laws, and reads
\be
\label{burger}
\del_tv + \del_x(v^2/2) = 0,
\ee
with $v=v(t,x)$, $t>0$ and $x\in \RR$.
This equation can be formally deduced from the Euler system of compressible fluids
$$
\aligned
&\del_t \rho + \del_x(\rho v) = 0, \\
&\del_t( \rho v) + \del_x(\rho v^2 + p(\rho)) = 0,
\endaligned
$$
where $p(\rho)$ denotes the pressure of the fluid with $\rho$ is the
density. By assuming a pressureless fluid  $p(\rho) \equiv 0$ and keeping a suitable
combination of the two equations, we can obtain \eqref{burger}. Namely, the following formal computation holds: 
$$
\aligned
0&= v\,\del_t( \rho) +  \rho\, \del_t( v) +   v^2 \del_x(\rho) +  2v \rho \, \del_x(v) \\
&=  \rho( \del_tv +2v\del_xv) + v ( \del_t \rho + v\del_x \rho) \\
&=  \rho( \del_tv +2v\del_xv) - v \rho \del_x v =  \rho( \del_tv
+v\del_x v).\endaligned
$$
Provided the density does not vanish, we thus get $\del_tv +v\del_x v=0$, 
which is equivalent to \eqref{burger}.

The  relativistic Burgers equation on  a flat  spacetime can be derived either by
imposing the Lorentz invariance property or formally from the Euler system on a curved background.
More precisely, the relativistic Burgers equation derived in \cite {LMO} on a flat background described by the Minkowski  metric in spherical coordinates $(t,r,\theta, \varphi)$
$$
g = -c^2dt^2 + dr^2+r^2d\theta^2 + r^2 \sin^2 \theta d\varphi^2, 
$$
reads 
\be
\label{RB}
\del_{t}  v + \del_{r}\big( {1 /{\eps^2}} \big( -1 + \sqrt{1+ \eps^2 v^2}  \big)\big) =0,
\ee
where $\eps$ is the inverse of the light speed. 

On the other hand, starting from the Euler system for relativistic compressible fluids and imposing vanishing pressure, we arrive at the following version of the non-relativistic and relativistic Burgers equations on Schwarzschild spacetime:
\be
\label{BES}
\del_t (r^2 {v}) + \del_r\Big(r(r-2m){{v^2 \over 2}}\Big) = {r}v^2-mc^2,
\ee
\be
\label{BEST}
\del_t(r^2 v) + \del_r \Big( r(r-2m)  \Big(-1+\sqrt{1+v^2}\Big)\Big)=0,
\ee
where the Schwarzschild metric in coordinates $(t,r,\theta, \varphi)$ is defined by
$$
g =-\Big(1-{2m\over r}\Big) c^2dt^2 + \Big(1-{2m \over r}\Big)^{-1}dr^2 + r^2 (d\theta^2 + \sin ^2 d\varphi^2),
$$
so that $m>0$ is the mass parameter, $c$ is the light speed, $r$ is
the Schwarzschild radius and  $r>2m$. We refer the reader to
\cite{LMO} for further details. In the present work, our main
objective is the discussion of yet another generalization, that is
the relativistic Burgers equation on  a
Friedmann--Lema\^{i}tre--Robertson--Walker (FLRW) spacetime.

%=================================================================

\section{FLRW background spacetimes}

\subsection*{Motivations from cosmology}

Cosmology is based on Einstein's theory of gravity and certain classes of explicit solutions are often considered. (See for instance \cite{HEL} for the notions in this section.)
Recall first that Einstein himself introduced in his field equation
the so-called cosmological constant $\Lambda$, in order to ensure that static solutions representing a static universe exist. Next, without requiring this cosmological constant,
Friedmann discovered solutions to Einstein equations describing an expanding universe. At the same time,  Lema\^{i}tre proposed the ``Big Bang model", which describes an expanding universe from a singular state and derived the ``distance redshift" relation. This circle of ideas, together with further works by Robertson and  Walker, led to a theory based on a family of solutions, now referred as
 the Friedmann--Lema\^{i}tre--Robertson--Walker spacetimes describing the whole universe evolution.

In short, the cosmological principle states that the universe is \textsl{ homogeneous} (has spatial translation symmetry) and \textsl{isotropic} (has spatial rotation symmetry). According to this principle, the universe may evolve in time, in either a contracting or an expanding direction. Observations  indicate that the universe is {\sl expanding}; whereas galaxies, quasars and galaxy clusters evolve with redshift, and the temperature of the cosmic microwave background (a uniform background of radio waves which fill the universe) is decreasing. An important feature in cosmology works is that studies are always done in {\sl co-moving coordinates} which expand with the universe. Furthermore, three topologies (positive, negative, or vanishing curvature) are possible and the universe is referred to be closed, open, or flat, respectively.

%------------------------------------------------------------------------------------------------------------------------------------

\subsection*{Expression of the FLRW metric}

We will work here with the FLRW metric describing a spatially homogeneous and isotropic three-dimensional space. In term of the proper time $t$ measured by a co-moving observer, and by introducing radial $r$ and angular ($\theta$ and $\varphi$) coordinates in the co-moving frame, we can express the metric of such a $3+1$-dimensional spacetime in the form
\be
\label{flrw}
g =-c^2dt^2+a(t)^2\Big({dr^2 \over 1-kr^2}+r^2d\theta^2+r^2\sin^2\theta d\varphi^2\Big),
\ee
where $k=0, \pm 1$. The variable $t$ is the proper time experienced by co-moving observers, who remain at rest in co-moving coordinates $dr=d\theta=d\varphi=0$. The time variable $t$ appearing in the FLRW metric is the time that would be measured by an observer who sees uniform expansion of the surrounding universe; it is named as the {\it cosmological proper time} or cosmic time.

The function $a$ reads $a(t)=a_0\Big({t\over t_0}\Big)^\alpha$, 
where, for the FLRW metric,  $\alpha={2 \over 3}$, $t_0$ is the age of the universe (which is a `large' number) and $a_0=1$  refers to `today'.
In addition, the parameter $k$,  a constant in time and space, is related to the spacetime curvature $K$  by the relation $k={a(t)}^2K$. 
We can distinguish between three cases:
\be
k = \begin{cases}1, \quad  \, \mbox {sphere (of positive curvature)},\\
0, \quad \, \mbox {(flat) Euclidean space, }\\
-1, \,\,  \mbox {hyperboloid (of negative curvature)}. \end{cases}
\ee

The FLRW metric can also be used to express the line element for
homogeneous, isotropic spacetime in matrix form  as
$$
g =g_{ij}dx^idx^j= (dt\, dr\, d\theta \, d\varphi)\,
\begin{pmatrix}
-c^2 & 0 & 0 & 0  \\
 0 & {a^2\over 1-kr^2} & 0 & 0 \\
 0 & 0& a^2r^2 & 0 \\
 0 & 0& 0 & a^2r^2\sin^2\theta
\end{pmatrix}
\,
\begin{pmatrix}
dt\\
dr\\
d\theta\\
d\varphi
\end{pmatrix}.
$$
Thus, the FLRW metric is diagonal with \be \label{g1} g_{00}=-c^2,
\, g_{11}={a^2\over1-kr^2}, \, g_{22}=a^2r^2, \, g_{33}=
a^2r^2\sin^2\theta, \ee as its non-zero covariant components, and
the corresponding contravariant components are \be \label{g2}
g^{00}=-{1\over c^2}, \, g^{11}={1-kr^2\over a^2}, \, g^{22}={1\over
a^2r^2}, \, g^{33}= {1\over a^2r^2\sin^2\theta}, \ee with
$
g^{ik}g_{kj}=\delta_j^i,
$
where $\delta_j^i$ is Kronecker's delta function.  We normalize
light speed parameter ($c=1$) in the FLRW metric for simplification
so that $g_{00}=g^{00}=-1$. The coordinates $(r,\theta,\varphi)$ of
the metric are co-moving coordinates. In the FLRW metric, as the
universe expands the galaxies keep the same coordinates
$(r,\theta,\varphi)$ and only the scale factor $a(t)$ changes with
time.
%--------------------------------------------------------------------------------------------
%--------------------------------------------------------------------------------------------------------------------
\subsection*{Christoffel symbols for FLRW background}

We need first to calculate the Christoffel symbols
$\Gamma_{\alpha\beta}^\mu$. The metric tensors tell us how to define distance between neighbouring points and the connection coefficients tell us how to define parallelism between neighbouring points.
We calculate the Christoffel symbols by using \eqref{g1} and \eqref{g2} with
$
\Gamma_{\alpha\beta}^\mu={1\over 2}g^{\mu \nu}(-\del_\nu  g_{\alpha\beta}+ \del_\beta g_{\alpha\nu}+ \del_\alpha g_{\beta\nu}),
$
where $\alpha, \beta, \mu, \nu \in \{0,1,2, 3\}$.  To begin with, we calculate two typical
coefficients by using \eqref{g1} and \eqref{g2}, as follows:
$$
\Gamma_{00}^0={1\over 2}g^{00}(-\del_0  g_{00}+ \del_0 g_{00}+ \del_0 g_{00})={1 \over 2}(-1)(0)=0,
$$
and
$$
\aligned
\Gamma_{11}^0
& ={1\over 2}g^{00}(-\del_0  g_{11}+ \del_1 g_{10}+ \del_0 g_{10})
\\
& ={1\over 2}(-1)\big(-\del_0({a^2\over (1-kr^2)})\big)={a \dot a \over c(1-kr^2)}.
\endaligned
$$
Similarly, we obtain the other non-vanishing Christoffel symbols as:
\be
\aligned
&\Gamma_{11}^0={ a \dot a\over c(1-kr^2)}, \quad
\Gamma_{22}^0={a \dot a r^2\over c}, \quad
\Gamma_{33}^0={a \dot a r^2 \sin^2\theta \over c},
 \\
 &\Gamma_{11}^1={kr\over 1-kr^2}, \quad
\Gamma_{22}^1=-r(1-kr^2), \quad
\Gamma_{33}^1= -r(1-kr^2)\sin^2 \theta, \\
&\Gamma_{33}^2= -\sin \theta \cos \theta,\quad
\Gamma_{23}^3=\Gamma_{32}^3=\cot \theta,
\quad
\Gamma_{12}^2=\Gamma_{21}^2=\Gamma_{31}^3=\Gamma_{13}^3={1\over r}, \\
&\Gamma_{01}^1=\Gamma_{10}^1=\Gamma_{02}^2=\Gamma_{20}^2=\Gamma_{30}^3=\Gamma_{03}^3= {\dot a\over ca}.
\endaligned
\ee

%==========================================================================

\section{From the Euler system to the relativistic Burgers equation}

\subsection*{The energy-momentum tensor for perfect fluids}

We assume that solutions to the Euler equations depend only on the time
variable $t$ and the radial variable $r$, and that the non-radial components of the velocity vanish, that is, $(u^\alpha) = (u^0(t,r), u^1(t,r), 0,0)$. Since $u$ is unit vector,
we have $u^{\alpha} u_{\alpha}=-1$ and we can write
$$
u^{\alpha} u_{\alpha}=u^0 u_0+u^1 u_1=g_{00}(u^0) (u^0)+g_{11}(u^1) (u^1),
$$
which gives us
\be
-1= g_{00}(u^0)^2+g_{11}(u^1)^2.
\ee
Plugging the covariant components into this equation, it follows that
\be
\label{unit}
-1 = - (u^0)^2+{a(t)^2\over 1-kr^2}(u^1)^2.
\ee
The  coordinates are taken to be $(x^0, x^1, x^2, x^3) = (c t,r, 0, 0)$. I is convenient to introduce the velocity component  
\be
\label{velocity}
v :=  {ca(t) \over (1-kr^2)^{1/2}} {u^1 \over u^0}.
\ee
By using  \eqref{unit}  and \eqref{velocity} with a simple algebraic manipulation, we obtain the following  identities
\be
\label{u}
(u^0)^2 = {c^2 \over (c^2-v^2) }, \qquad
\quad
(u^1)^2 ={v^2(1-kr^2) \over a^2(c^2-v^2)}.
\ee
Then, in order to calculate the tensor components, we need to recall the   energy momentum tensor of perfect fluids formula, namely
\be
\label{tens}
T^{\alpha\beta} = (\rho c^2 + p) \, u^\alpha \, u^\beta + p \, g^{\alpha\beta}.
\ee
By inserting the terms from the relation \eqref{u} and the contravariant components \eqref{g2} into the formula \eqref{tens}, we obtain the components of the energy momentum tensor. For example, we have 
$$
T^{00} = (\rho c^2 + p) u^0  u^0 + p  g^{00}={c^2 \over c^2-v^2}(\rho c^2+p)-p={ \rho c^4+pv^2 \over c^2-v^2}.
$$
In the same way  the other non-vanishing components are 
 \[
\begin{array}{lll}
\displaystyle  T^{01} = T^{10} = {cv(1-kr^2)^{1/2}( \rho c^2+p) \over a(c^2-v^2)}, & \displaystyle T^{11} = {c^2(1-kr^2)(v^2 \rho +p) \over a^2(c^2-v^2)},\\
\displaystyle T^{22} = {p \over a^2r^2}, & \displaystyle T^{33} = {p \over a^2r^2 \sin^2\theta}. &
\end{array}
\] 

%------------------------------------------------------------------------------------

\subsection*{The pressureless Euler system on   FLRW background}

In the previous section, Christoffel symbols and energy momentum tensors for perfect fluids
were derived. In this section, we are in a position to derive the Euler system on a FLRW spacetime. We recall the Euler equations  $\nabla_\alpha T^{\alpha\beta} =0,$  which can be rewritten as
\be
\label{wrt}
\del_\alpha T^{\alpha\beta} +\Gamma_{\alpha\gamma}^\alpha  T^{\gamma\beta}+\Gamma_{\alpha\gamma}^\beta  T^{\alpha\gamma}=0.
\ee
There are two sets of equations depending on $\beta$. Firstly taking  $\beta=0$ in \eqref{wrt} yields
$$
\del_\alpha T^{\alpha 0} +\Gamma_{\alpha\gamma}^\alpha  T^{\gamma 0}+\Gamma_{\alpha\gamma}^0  T^{\alpha\gamma}=0,
$$
which is equivalent to
$$
\aligned
&\del_0 T^{0 0}+\Gamma_{0 \gamma}^0  T^{\gamma 0}+\Gamma_{ \gamma 0}^0  T^{\gamma 0}+\del_1 T^{1 0}+\Gamma_{1 \gamma}^1  T^{\gamma 0}+\Gamma_{1 \gamma}^0  T^{1 \gamma }+\del_2 T^{2 0}+\Gamma_{2 \gamma}^2  T^{ \gamma 0}\\
&+\Gamma_{2 \gamma}^0  T^{2 \gamma }+\del_3 T^{3 0}+\Gamma_{3 \gamma}^3  T^{ \gamma 0 }+\Gamma_{3 \gamma}^0  T^{3 \gamma}=0.
\endaligned
$$
We next consider the exponent $\beta=1$, that is, 
$
\del_\alpha T^{\alpha 1} +\Gamma_{\alpha\gamma}^\alpha  T^{\gamma 1}+\Gamma_{\alpha\gamma}^1  T^{\alpha\gamma}=0,
$
which gives us
$$
\aligned
&\del_0 T^{0 1}+\Gamma_{0 \gamma}^0  T^{\gamma 1}+\Gamma_{ 0 \gamma}^0  T^{0 \gamma }+\del_1 T^{1 1}+\Gamma_{1 \gamma}^1  T^{\gamma 1}+\Gamma_{1 \gamma}^1  T^{1 \gamma}+\del_2 T^{2 1}+\Gamma_{2 \gamma}^2 T^{ \gamma 1}\\
&+\Gamma_{2 \gamma}^1  T^{2 \gamma}+\del_3 T^{3 1}+\Gamma_{3 \gamma}^3  T^{ \gamma 1 }+\Gamma_{3 \gamma}^1  T^{3 \gamma}=0.
\endaligned
$$

Next, by substituting the expression of the Christoffel symbols in the  Euler system on a FLRW background, we obtain  the simplified system
\be
\aligned
\label{E.E.0}
&\del_0 T^{00} + \del_1 T^{10}+{3 \dot a \over ca}T^{00}+{ k r \over 1-kr^2}T^{10}+{ a \dot a \over c(1-kr^2)}T^{11}\\
&+{2\over r}T^{10}+{r^2a \dot a\over c}T^{22}+ {a \dot a r^2 \sin^2\theta\over c}T^{33}= 0,\\
&\del_0 T^{01} + \del_1 T^{11}+{4 \dot a \over ca}T^{01} +{ \dot a \over ca}T^{10}+{ 2kr \over (1-k r^2)}T^{11}\\
&+{ 1\over r}T^{11}-r(1-kr^2)T^{22}-r(1-kr^2) \sin^2\theta T^{33}= 0.
\endaligned
\ee
Finally, using the expressions for perfect fluids into  \eqref{E.E.0} and assuming that the pressure $p$ vanishes identically, we obtain the Euler system on a FLRW background:
\be
\aligned
\label{EE1}
&\,\,\, \del_0 \left({ \rho c^2\over c^2-v^2} \right)+\del_1 \left({ \rho cv(1-kr^2)^{1/2}\over a(c^2-v^2)} \right)+{3 \dot a \rho c\over a(c^2-v^2)}+{2 \rho cv(1-kr^2)^{1/2}\over ra(c^2-v^2)}\\
&\hskip5.cm  +{k r \rho cv\over a(c^2-v^2)(1-kr^2)^{1/2}}+ {\dot a v^2 \rho\over ca(c^2-v^2)} = 0,
\endaligned
\ee
\be
\aligned
\label{EE2}
&\del_0\left({c^2 \rho v (1-kr^2)^{1/2}\over a(c^2-v^2)} \right)
 +  \del_1 \left({cv^2 \rho (1-kr^2)\over a^2(c^2 -v^2)}\right) +{5 \dot a \rho v c (1-kr^2)^{1/2}\over a^2(c^2-v^2)}\\
&\hskip5.cm +{2krcv^2 \rho\over a^2(c^2-v^2)}+{2cv^2 \rho (1-kr^2)\over ra^2(c^2-v^2)}= 0.
\endaligned
\ee

%============================================================
\section{The relativistic Burgers equation on a FLRW background}
%------------------------------------------------------------------------------------

\subsection*{The derivation of the relativistic Burgers equation}

We now explain how to formal derive the relativistic Burgers equation from  the fluid equations \eqref{EE1}-\eqref{EE2}.  
First of all, from \eqref{EE1}, we express $\displaystyle \del_0
\left({ \rho c^2\over c^2-v^2} \right)$ as
$$
\aligned
%\label{EE3}
\del_0 \left({ \rho c^2\over c^2-v^2} \right)=
& -\del_1 \left({ \rho cv(1-kr^2)^{1/2}\over a(c^2-v^2)} \right)-{3 \dot a \rho c\over a(c^2-v^2)}-{2 \rho cv(1-kr^2)^{1/2}\over ra(c^2-v^2)}\\
&  -{k r \rho cv\over a(c^2-v^2)(1-kr^2)^{1/2}}- {\dot a v^2 \rho\over ca(c^2-v^2)}.
\endaligned
$$
Next, we take partial derivatives  \eqref{EE2} and get 
$$
\aligned
%\label{EE4}
&\del_0\left({c^2 \rho \over c^2-v^2} \right)\left({v (1-kr^2)^{1/2}\over a} \right)+\left({c^2 \rho \over c^2-v^2} \right)\del_0\left({v (1-kr^2)^{1/2}\over a} \right)\\
&\,+\del_1 \left({cv \rho (1-kr^2)^{1/2}\over a(c^2 -v^2)}\right) \left({v(1-kr^2)^{1/2}\over a}\right)+\left({cv \rho (1-kr^2)^{1/2}\over a(c^2 -v^2)}\right) \del_1 \left({v(1-kr^2)^{1/2}\over a}\right)\\
&\,+{5 \dot a \rho v c (1-kr^2)^{1/2}\over a^2(c^2-v^2)}+{2krcv^2
\rho\over a^2(c^2-v^2)}+{2cv^2 \rho (1-kr^2)\over ra^2(c^2-v^2)}= 0.
\endaligned
$$
We substitute the expression $\displaystyle\del_0 \left({
\rho c^2\over c^2-v^2} \right)$ and find 
\be
\aligned
\label{EE5}
& \del_1 \left({ \rho cv(1-kr^2)^{1/2}\over a(c^2-v^2)} \right)+{3 \dot a \rho c\over a(c^2-v^2)}+{2 \rho cv(1-kr^2)^{1/2}\over ra(c^2-v^2)}+{k r \rho cv\over a(c^2-v^2)(1-kr^2)^{1/2}}\\
&\qquad +{\dot a v^2 \rho\over ca(c^2-v^2)}\Bigg\}\left({v (1-kr^2)^{1/2}\over a} \right)+\left({c^2 \rho \over c^2-v^2} \right)\del_0\left({v (1-kr^2)^{1/2}\over a} \right)\\
&\qquad +\del_1 \left({cv \rho (1-kr^2)^{1/2}\over a(c^2 -v^2)}\right) \left({v(1-kr^2)^{1/2}\over a}\right)+\left({cv \rho (1-kr^2)^{1/2}\over a(c^2 -v^2)}\right) \del_1 \left({v(1-kr^2)^{1/2}\over a}\right)\\
&\qquad +{5 \dot a \rho v c (1-kr^2)^{1/2}\over
a^2(c^2-v^2)}+{2krcv^2 \rho\over a^2(c^2-v^2)}+{2cv^2 \rho
(1-kr^2)\over ra^2(c^2-v^2)}= 0.
\endaligned
\ee
After further straighforward calculations and replacing $(\del_t,\del_r)$ by
$(\del_0,\del_1)$, we reach 
\be 
\label{maineq} 
a^2\del_t ({v \over
a}(1-kr^2)^{1/2})+\del_r( ({v^2\over 2})(1-kr^2))+v(1-kr^2)^{1/2}a_t
(2-{v^2\over c^2})+rkv^2=0. 
\ee 
It follows that
$$
(av_t-va_t)(1-kr^2)^{1/2}+(1-kr^2)\del_r({v^2\over 2})-rkv^2+v(1-kr^2)^{1/2}a_t(2-{v^2\over c^2})+rkv^2=0,
$$
and thus,  after simplification, 
$$
av_t(1-kr^2)^{1/2} +(1-kr^2)\del_r({v^2\over 2})+v \Big( 1-kr^2 \Big)^{1/2}a_t(1-{v^2\over c^2})=0.
$$
Finally, we arrive at the following definition.

\begin{definition}
The {\bf relativistic Burgers equation on a FLRW background} is
\be
\label{lasteq}
a \, v_t + \big(1-kr^2 \big)^{1/2} \del_r \Big( {v^2 \over 2}\Big) + v \Big( 1-{v^2\over c^2} \Big) \, a_t = 0,
\ee
in which $a=a(t)>0$ is a given function,
$k \in \big\{-1, 0, 1\big\}$ is a discrete parameter,
and the light speed $c$ is a positive parameter.
\end{definition}

In the limiting case  $c \to  +\infty$, the equation \eqref{lasteq} can be rewritten as
\be
\label{good1}
 \del_t  \Big( {a(t)v \over (1-kr^2)^{1/2}}\Big)+ \del_r \Big( {v^2 \over 2}\Big) = 0,
\ee
which is a {\sl conservation law}.

In order to obtain an analogous equation from \eqref{lasteq} for finite $c$ values, we propose to rewrite \eqref{lasteq} as
\be
\label{good2}
 \del_t  \Big( {a(t)v \over (1-kr^2)^{1/2}}\Big)-{v^3 \over c^2} \del_t  \Big( {a(t) \over (1-kr^2)^{1/2}}\Big)+ \del_r \Big( {v^2 \over 2}\Big) = 0.
\ee
Furthermore, in the special case $a(t) \equiv 1$ for the equation \eqref{lasteq}, this latter equation is also a {\sl conservation law}
\be
\label{good3}
 \del_t  \Big( {v \over (1-kr^2)^{1/2}}\Big)+ \del_r \Big( {v^2 \over 2}\Big) = 0.
\ee

%--------------------------------------------------------------------------------------------------------------------------------------------

\subsection*{The initial value problem}

The equation \eqref{lasteq} is a nonlinear hyperbolic equation with time- and space-dependent coefficients.
The solutions admit jump discontinuities which propagate in time.
This equation fits in the general theory of entropy weak solutions to such equations by Kruzkov \cite{SNK}.
The notion of entropy solutions relies on the use of the so-called convex entropy pairs, defined as follows.

\begin{definition} A pair of Lipschitz continuous functions
$V, F$ is a convex entropy-entropy flux pair if  $V=V(v)$ is strictly convex and
$F':=vV'$ hold almost everywhere. A function  $v \in L^\infty(\mathbb{R}^+ \times \mathbb{R}^+ )$    is called an entropy solution of \eqref{lasteq},  if
for every convex entropy-entropy flux pair $(V,F)$
\be
\label{weak}
\aligned
&av_t+(1-kr^2)^{1/2}\del_r\big({v^2\over 2}\big)+v\big(1-{v^2\over c^2}\big)a_t=0,\\
&a V(v)_t+(1-kr^2)^{1/2}\del_r F(v)+v V'(v)\big(1-{v^2\over c^2}\big)a_t \leq 0,
\endaligned
\ee
hold in the sense of distributions.
\end{definition}

In view of the general theory in \cite{SNK}, we obtain the following.

\begin{theorem}
The equation \eqref{lasteq} admits an entropy weak solution  $v \in L^\infty(\mathbb{R}^+ \times \mathbb{R}^+ )$ satisfying the conditions \eqref{weak}  in the sense of Kruzkov's theory.
\end{theorem}

Note in passing that, in the particular case $a(t)\equiv 1$ and $k \equiv 0$, we obtain the classical Burgers equation and the approximate solution of this equation satisfies the additional estimate
$$
\inf_xv(0,x)\leq\inf_xv(t,x)\leq\sup_xv(t,x)\leq\sup_xv(0,x).
$$
This is of course not true in general, and the lack of such properties in one of the challenges in order to numerically cope with discontinuous solutions to \eqref{weak}.

%============================================================

\section{Special solutions and non-relativistic limit}

\subsection*{Spatially homogeneous solutions}

We look for special classes of explicit solutions to Burgers equation on a FLRW  background \eqref{lasteq}, which involves the variable coefficients $a(t)$ and $a_t(t)$. Due to this $t$-dependency, it is easily checked that for all three values of $k$, there does not exist any static solution (except $v \equiv 0$).

On the other hand, in order to  find spatially homogeneous solutions of \eqref{lasteq}, we assume that $v$ depends only on $t$ so that the term $\del_r({v^2\over 2})$ vanishes identically, which means
\be
\label{space}
av_t+v(1-{v^2\over c^2})a_t=0.
\ee
By changing the notation $v_t$ to $v'$, and $a_t$ to $a'$, we write
$
{v'\over v(1-{v^2\over c^2})}=-{a'\over a},
$
which is equivalent to
$$
\Big( {1\over v}+{{v\over c^2}\over 1-{v^2\over c^2}} \Big) \, v'=-(\log a)'.
$$
It follows that
$
{\pm v\over \sqrt{1-v^2/c^2}}={w\over a}$, 
where $w$ is a constant.
Equivalently, we have
$
{a^2\over w^2}v^2=1-{v^2\over c^2}.
$
Thus the spatially homogeneous homogeneous solutions can be described by the explicit formula
\be
\label{E.E.2}
v(t)={\pm c\over \sqrt{1+{a^2(t)c^2\over w^2}}},
\ee
 where $w$ is a constant  parameter.
This is obviously true for all $k \in\{-1,0,1 \}$.

%-------------------------------------------------------------------------------------------------------------------------------------------------------------------

\begin{proposition}

The spatially homogeneous solutions to the relativistic Burgers equation on a FLRW background
\be
\label{sps}
v(t) = {w\over \sqrt{a(t)^2+{w^2\over c^2}}} \in (-c,c)
\ee
 are parametrized by a real parameter $w$ (where $c$ is the light speed).
\end{proposition}

%--------------------------------------------------------------------------------------------------------------------------------------------------------------------

\subsection*{Some limit properties of the relativistic Burgers equation}

Next, let us consider some limit properties of the equation \eqref{lasteq} when, for definiteness,  $a(t)=a_0\big({t\over t_0}\big)^\alpha$.
Observe in passing that \eqref{lasteq} is {\sl not linear in terms of the coefficient $a(t)$}
(since the second term in the equation does not include $a(t)$ or $a'(t)$).
Recall that the following parameters are relevant:
$$
\begin{cases}
\mbox{$k$: curvature constant }, \quad k \in [-1,1]\\
\mbox{$c$:  light speed }, \quad c \in (0,\infty)\\
\mbox{$a_0$: constant in the scale factor $a(t)$}, \quad a_0 \in (0,\infty)\\
\mbox{$\alpha$: exponent in the scale factor $a(t)$}, \quad  \alpha \in (0,\infty). \\
\end{cases}
$$
Two typical ranges of the time variables are relevant here, since shock wave solutions to
 nonlinear hyperbolic equations are {\sl only defined in a forward time directions:}
since at $t=0$ the equation is singular, we can treat the range $t \in [1,\infty)$ or the range
 $t \in [-1,0)$.
 For $t \in [1,\infty)$ we normalize $a_0=1$ and for $t \in [-1,0)$ we set $a_0=-1$.

In the case $t>1$, if we consider the limit $t\to +\infty,$ the equation is {\sl expanding toward the future time directions}, while
in the case $t<0$ when $t\to 0$, the equation is {\sl contracting in the future time directions}.

%-------------------------------------------------------------------------------------------------------------------------------------------------------------------------

\subsubsection*{Recovering the standard Burgers equation}

The special  case $a_0=1$,  $t_0=1$, $\alpha=0$ (which means $a(t)=1$), with the particular case $k=0$ for the equation \eqref{lasteq} leads us to
\be
\del_tv + \del_r({v^2\over 2}) = 0,
\ee
which is  the classical Burgers equation.

%-------------------------------------------------------------------------------------------------------------------------------------------------------------------------

\subsubsection*{The non-relativistic limit}

As mentioned earlier, by taking the limit  $c\to  +\infty$ in the equation \eqref{lasteq}, we obtain
\be
\del_t(av) +(1-kr^2)^{1/2} \del_r({v^2\over 2}) = 0.
\ee
We can also determine directly the limiting behaviour of the spatially homogeneous solutions to \eqref{lasteq}:  in view of \eqref{E.E.2}, we obtain
\be
\label{SI}
v(t)={1\over \sqrt{{1\over c^2}+{a^2(t)\over w^2}}},
\ee
where $w$ is a constant parameter. Here we have made the following observations:

\begin{itemize}

\item For spatially homogeneous solutions, we have $|v|<c$.

\item  In the expanding direction $t\to  +\infty$, we have  $v \to  0$.

\item  In the contraction direction $t\to  0$,   we have $v \to  c$ since   $a(t) \to  0$.

\item  We have $v\to  {w\over a(t)}$  as  $c\to  +\infty$.
\end{itemize}

%=====================================================================

\section{The finite volume method} 

\subsection*{Finite volume methodology for geometric balance laws}

In this section, we are motivated by the earlier works \cite{Russo1,Russo2} for nonlinear hyperbolic problems without relativistic features and \cite{LMO} concerning relativistic Burgers equations.
In Burgers equation on a FLRW background, the variable coefficients {\sl depend upon the time variable}
$t$, due to  the terms $a(t)$, $a'(t)$ and $k\in\{ -1,0,1\}$. Hence, the numerical approximation of solutions to Burgers equation on a FLRW background leads to a new challenge, in comparison with
flat or Schwarzschild backgrounds.

As explained earlier, the spacetime of interest is described by a single chart and some coordinates denoted by $(t,r)$. For the discretization, we denote the (constant) time length by $\Delta t$ and we set $t_n=n \Delta t$,
 and we introduce equally
spaced cells $I_j=[r_{j-1/2},r_{j+1/2}]$ with (constant) spatial length denoted by $\Delta r=r_{j+1/2}-r_{j-1/2}$. The finite volume method is based on an averaging of the balance law
\be
\label{finite1}
\del_t(T^0(t, r))+\del_r(T^1(t, r))=S(t,r),
\ee
over each grid cell $[t_n,t_{n+1}] \times I_j$, where   $ T^{\alpha}(v)=T^{\alpha}(t,r)$ and $S(t,r)$
are the flux and source terms, respectively. We thus have the identity
$$
\aligned
&\int_{r_{j-1/2}}^{r_{j+1/2}} (T^0(t_{n+1},r)-T^0(t_n,r))\,dr \\
&+\int_{t_n}^{t_{n+1}}(T^1(t,r_{j+1/2})-T^1(t,r_{j-1/2}))\,dt
= \int_{[t_n,t_{n+1}] \times I_j}S(t,r)\,dt\,dr
\endaligned
$$
or, by re-arranging the terms,
\be
\aligned
\label{finite3}
\int_{r_{j-1/2}}^{r_{j+1/2}} T^0(t_{n+1},r)\,dr=\int_{r_{j-1/2}}^{r_{j+1/2}} T^0(t_n,r)dr+\int_{[t_n,t_{n+1}] \times I_j}S(t,r)\,dt\,dr \\
-\int_{t_n}^{t_{n+1}}(T^1(t,r_{j+1/2})-T^1(t,r_{j-1/2}))\,dt.
\endaligned
\ee
Then, we introduce the following approximations
$$
\aligned
&  {1\over \Delta r}\int_{r_{j-1/2}}^{r_{j+1/2}} T^0(t_n,r)\,dr \simeq \overline{T}_j^n,
\qquad 
{1\over \Delta t}\int_{t_n}^{t_{n+1}}T^1(t,r_{j \pm 1/2}) \,dt \simeq \overline Q_{j \pm 1/2}^n,
\\
& {1 \over \Delta t \Delta r}\int_{[t_n,t_{n+1}] \times I_j}S(t,r)\,dt\,dr \simeq \overline{S}_j^n.
\endaligned
$$
so that our scheme take the following finite volume form
\be
\label{well1}
\overline{T}_j^{n+1}=\overline{T}_j^n-{\Delta t\over \Delta r}(\overline Q_{j +1/2}^n- \overline Q_{j -1/2}^n)+\Delta t  \overline{S}_j^n.
\ee
Keeping in mind the practical implementation of the scheme, we write also $\overline{T}_j^n=\overline{T}(v_j^n)$, where $\overline{T}$ is the (invertible) map determined by the equation. The piecewise constant approximations
$(v_j^n)$ at the ``next'' time level are thus given by the formula
\be
v_j^{n+1}=\overline{T}^{-1}
\Big( \overline{T}(v_j^n)-{\Delta t\over \Delta r}( \overline Q_{j +1/2}^n- \overline Q_{j -1/2}^n)+\Delta t  \overline{S}^{n}_j \Big).
\ee
For the scheme to be fully specified, we still need to select a numerical flux and an approximation of the source term.

%===============================================================================

\subsection*{Approximating Burgers equation on a FLRW background}

Consider the partial differential equation 
\be
\label{ww1}
 \del_t v + \del_r f(v,r) = 0,
\ee
for which time-dependent solutions have the property that $r \mapsto f(v(r),r)$ is constant in $r$. A general finite volume approximation for this equation \eqref{ww1} can be written as
\be
\label{sch}
v_j^{n+1}=v_j^n-{\Delta t \over \Delta r}(f _{j+1/2}^n  -f _{j-1/2}^n ).
\ee
Considering \eqref{ww1} together with \eqref{sch}, we have the following observing concerning the family of time-independent solutions.

\begin{claim}
If the initial flux-terms $f(v_j^0, r_j)$ are equal (that is, independent pf the spatial index $j$), then the scheme \eqref{sch} 
yields $ v_j^n = v_j^0$ which is thus independent of $n$.
\end{claim}

The relativistic Burgers equation  \eqref{lasteq} under consideration is more involved 
and can be put in two different forms. The non--conservative form reads 
\be
\label{form1}
v_t + \big(1-kr^2 \big)^{1/2}{1 \over a(t)} \del_r \Big( {v^2 \over 2}\Big) =- v ( 1-v^2 ) \, {a_t(t)\over a(t)},
\ee
where the source term is
$S_N=- v ( 1-v^2 ) \, {a_t(t)\over a(t)}$, while the conservative form reads 
\be
\label{form2}
\del_t v
+\del_r\Big((1-kr^2)^{1/2}{v^2\over 2a(t)}\Big)=-\Big({krv^2\over
2a(t)}(1-kr^2)^{-1/ 2}+v(1-v^2){a_t(t)\over a(t)}\Big),
\ee
with the source term
$S_C=-\Big({krv^2\over
2a(t)}(1-kr^2)^{-1/ 2}+v(1-v^2){a_t(t)\over a(t)}\Big)$. 
Note the obvious relation
$S_C= S_N + \widetilde S$
with 
$\widetilde S= - {krv^2\over 2a(t)}(1-kr^2)^{-1/ 2}$. 

The finite volume scheme in both cases has the general form 
\be
\label{form3}
v_j^{n+1}=v_j^n-{\Delta t \over \Delta r}(b_{j+1/2}^n g _{j+1/2}^n  - b_{j-1/2}^n g _{j-1/2}^n )+\Delta t S_j^n,
\ee
where $b_{j+1/2}^n g _{j+1/2}^n=f _{j+1/2}^n$ are the numerical flux functions. This is the form that our scheme will take.

%===============================================================================

\subsection*{Second-order Godunov--type scheme}

Any first order scheme for the equation
$
\del_t v + \del_r f(v,r) = 0
$
can be turned into a second-order method by advancing the
cell-boundary values which are used in the numerical flux functions
in order to determine the intermediate time level $t^{n + 1/2}=(t^n +t^{n+1})/2$.  More precisely, the second-order
Godunov scheme is obtained from the edge values of the reconstructed
profile advanced by half a time step. Following van Leer
\cite{L} and the textbook Guinot \cite{VG}, our
algorithm of the method is formulated as follows:
\begin{itemize}
\item We reconstruct the variable within the computational cells.
This couples of values $(v_{i,L}^n, v_{i,R}^n)$ in each
computational cell. We know that $v_{i,L}^n$ lies between
$v_{i-1}^n$ and $v_{i}^n,$ and $v_{i,R}^n$ lies between $v_{i}^n$
and $v_{i+1}^n$.
\item We proceed the solution by half a step in time. The intermediate values $v$ at the cell edges
at the time $t^{n + 1/2}=(t^n+t^{n+1})/2$ are denoted by
$(v_{i,L}^n,v_{i,R}^n)$. We calculate these values by
$$
v_{i,L}^{n+1/2}=v_{i,L}^n-{\Delta t \over 2\Delta
r}[f(v_{i,R}^n)-f(v_{i,L}^n)],
$$
$$
v_{i,R}^{n+1/2}=v_{i,R}^n-{\Delta t \over 2\Delta
r}[f(v_{i,R}^n)-f(v_{i,L}^n)].
$$
\item We next solve  the Riemann problem formed by the intermediate values $(v_{i,L}^n,v_{i,R}^n)$.
The solution $v_{i+1/2}^{n+1/2}$ is used to compute the flux
$f_{i+1/2}^{n+1/2}=f(v_{i+1/2}^{n+1/2})$.
\item Finally we proceed the solution by the time step $\Delta t$ from $t^n$ using the classical formula
$$
v_i^{n+1}=v_i^n-{\Delta t \over \Delta
r} \, \Big( f_{i+1/2}^{n+1/2}-f_{i-1/2}^{n+1/2} \Big).
$$
In the next section, according to this algorithm, we reconstruct our
numerical second-order scheme by considering the relativistic
Burgers equation including a source term.
\end{itemize}

%===============================================================================

\section{Numerical approximation of Burgers equation on a FLRW background}

\subsection*{Godunov scheme for Burgers equation}

In this part,  numerical experiments are illustrated  for the model derived on a FLRW spacetime based on a first order  Godunov scheme.  Mainly,  the behaviours of initial single shocks and rarefactions are examined in the numerical tests depending on three particular cases of constant $k$. 
Analogously, depending on the parameter $k$ in the main equation, we have several illustrations.

We analyze the given model  with a single shock and rarefaction for an  initial function considering the Godunov scheme with a local Riemann problem for each grid cell. In the experiments for test functions, we choose $a(t)= t^2$ and $r \in [-1,1]$.  Since our scheme has singularities at $t=0$ stemming from the function $a(t)$, we start by taking $t>1$ for all cases of $k= -1, 0, 1$. In Riemann problem both shocks and rarefaction waves are produced, thus we look for the fastest wave at each grid cell. We impose transmissive  boundary conditions on the scheme.
After normalization  (taking $c=1$) in the equation \eqref{lasteq}, we obtain the following model
\be
\label{godunov1}
\del_t v +(1-kr^2)^{1/2}{1\over a(t)}\del_r\Big({v^2\over 2}\Big)=-\Big(v(1-v^2){a_t(t)\over a(t)}\Big),
\ee
and the corresponding finite volume scheme is  written as
\be
\label{godunov3}
v_j^{n+1}=v_j^n-{\Delta t \over \Delta r}(b_{j+1/2}^n\  g _{j+1/2}^n  - b_{j-1/2}^n g _{j-1/2}^n )+\Delta t S_j^n,
\ee
where
$$
S_j^n=-\Big(v_j^n(1-(v_j^n)^2){a_t^n\over a^n}\Big),
$$
and
$$
\aligned
b_{j\pm1/2}^n&={ (1-k({r}_{j\pm1/2}^n)^2)^{1/2} \over a^n},\\
g _{j-1/2}^n&=f(v_{j-1}^n,v_{j}^n),\qquad
g _{j+1/2}^n&=f(v_j^n,v_{j+1}^n),
\endaligned
$$
with flux function $f(u,v)$  defined as follows
\be
\label{flux}
f(u,v)= \begin{cases}{u^2 \over 2}, \quad  \, \mbox {if} \qquad u>v \qquad   \mbox {and} \qquad u+v>0,\\
{v^2 \over 2}, \quad  \, \mbox {if} \qquad u>v \qquad   \mbox {and} \qquad u+v<0,\\
{u^2 \over 2}, \quad  \, \mbox {if} \qquad u\leq v \qquad   \mbox {and} \qquad u>0,\\
{v^2 \over 2}, \quad  \, \mbox {if} \qquad u\leq v \qquad   \mbox {and} \qquad v<0,\\
\,{0},\quad  \, \, \, \mbox{if}  \qquad u\leq v \qquad \mbox{and} \qquad  u \leq 0 \leq v.\\
              \end{cases}
\ee
For the sake of stability, we require that $\Delta t$ and $\Delta r$ satisfy 
$$
{\Delta t \over \Delta r} \underset{j}{\max}\Big|{(1-k(r_j^n)^2)^{1/2}v_j^n\over a^n}\Big| \leq 1. 
$$ 

We implemented the first-order Godunov scheme and studied the dynamics of 
shocks and rarefactions, by comparing between the 
cases $k=-1$, $k=0$ and $k=1$. These results are presented in  Figures $1$ and
$2$. From these graphs we observe that the numerical solution for
the particular case $k=-1$, which is represented by the red line,
moves faster than the particular case $k=0$ and  $k=1$, represented
by the green line and the blue line, respectively.  This  can
also be checked by plugging $k=-1,0,1$ into the speed term given by
\eqref{spd} which supports the theoretical background of the model
with the numerical results. We also observe that, for all particular
cases of $k=-1$, $k=0$  and $k=1$, the solution curves converge and
this yields the efficiency and robustness of the scheme.

Recall that   our model   admits  spatially homogenous solutions  described by equation  \eqref{sps} which yields
\be
\label{omega}
w = {v(t)a(t)\over \sqrt{1-v(t)^2}}.
\ee 

\begin{claim}
\label{kkk}
If $w$ given by \eqref{omega} remains constant, then the proposed scheme is well-balanced in the sense that that discrete forms of the spatially homogeneous solutions are computed exactly. 
\end{claim} 

\begin{proof} Discrete solutions satisfying the time-independence property 
$$
v_j^{n+1}=v_j^n
$$
are charcaterized by the condition 
$$
{\Delta t \over \Delta r}(b_{j+1/2}^n\  g _{j+1/2}^n  + b_{j-1/2}^n g _{j-1/2}^n )+\Delta t S_j^n = 0, 
$$
or equivalently, 

$$
{1\over \Delta r}(b_{j+1/2}^n\  g _{j+1/2}^n  + b_{j-1/2}^n g _{j-1/2}^n )
= 
-  v_j^n(1-(v_j^n)^2){a_t^n\over a^n},
$$
which by construction represent a discretization of the spatially homogeneous  solutions of interest. 
\end{proof} 
%===============================================================================

\subsection*{Second-order well-balanced Godunov scheme}

In order to increase the accuracy in the numerical experiments, we
construct a second-order well--balanced scheme based on Godunov
method. According to  the construction detailed in the previous
section for the second-order schemes, we write  the second-order finite volume
approximation for our model. The schemes for the intermediate values
and the proceeding solutions are formulated as follows

\be
\label{w1}
v_{j\pm 1/2}^{n+1/2}=v_{j\pm 1/2}^n-{\Delta t \over 2
\Delta r}(b_{j+1/2}^{n} g _{j+1/2}^{n}  - b_{j-1/2}^{n} g
_{j-1/2}^{n} )+{\Delta t \over 2} S_{j\pm 1/2}^{n}, \ee \be
\label{w2} v_j^{n+1}=v_j^n-{\Delta t \over \Delta
r}(b_{j+1/2}^{n+1/2} g _{j+1/2}^{n+1/2}  - b_{j-1/2}^{n+1/2} g
_{j-1/2}^{n+1/2} )+\Delta t S_j^{n+1/2},
\ee
where
$
t^{n + 1/2}=(t^n +t^{n+1})/2
$
and
$$
\aligned
b_{j\pm1/2}^{n+1/2}&={ (1-k({r}_{j\pm1/2}^{n+1/2})^2)^{1/2} \over a^{n+1/2}},\\
g _{j-1/2}^{n+1/2}&=f(v_{j-1}^{n+1/2},v_{j}^{n+1/2}),\qquad
g _{j+1/2}^{n+1/2}&=f(v_j^{n+1/2},v_{j+1}^{n+1/2}),
\endaligned
$$
with

$$
S_j^{n+1/2}=-\Big(v_j^{n+1/2}(1-(v_j^{n+1/2})^2){a_t^{n+1/2}\over a^{n+1/2}}\Big).
$$

In addition, the flux function function $f$ is given analogously by the relation
\eqref{flux}.

The implementation of the second-order Godunov scheme is based on
the construction given above. In order to compare the efficiency of
the first-order and second-order schemes, we repeat the numerical
tests for the second-order case which are already implemented  for
the first order schemes. We consider shocks and rarefactions  in
Figures $3$ and $4$ for three cases $k=-1$, $k=0$ and $k=1$  of the
second-order well-balanced Godunov scheme. We deduce that the
numerical solution for the particular case $k=-1$, represented by
the red line,  again moves faster than the particular cases $k=0$
and $k=1$, represented by the green line and the blue line,
respectively.
\begin{figure}
\centering
\includegraphics[width=5.5cm]{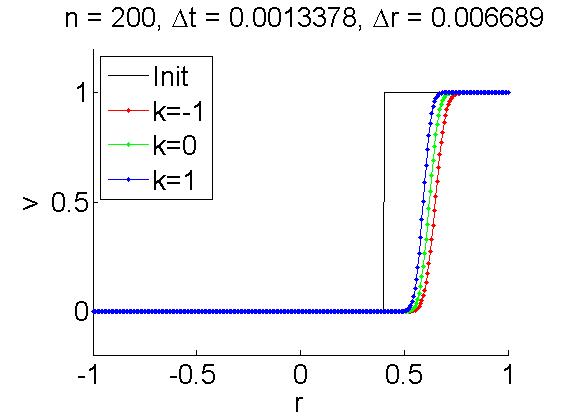} \includegraphics[width=5.5cm]{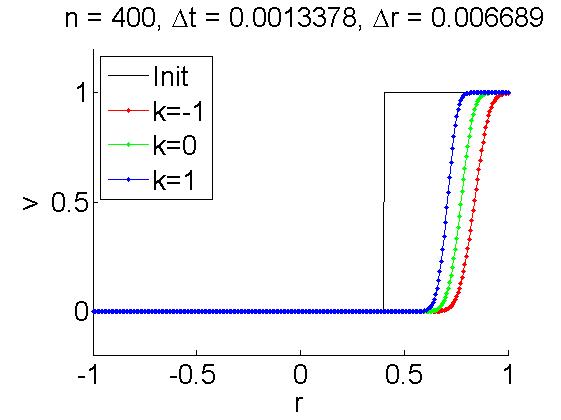} \\
\includegraphics[width=5.5cm]{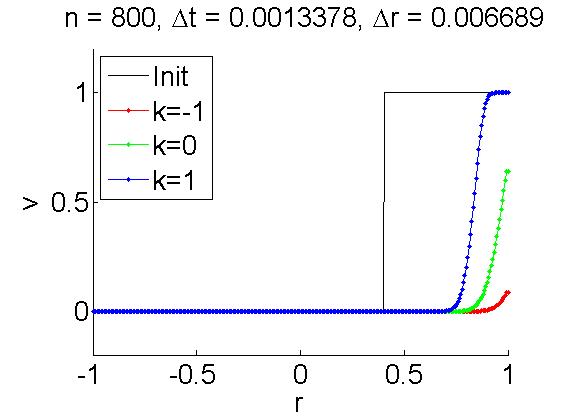} \includegraphics[width=5.5cm]{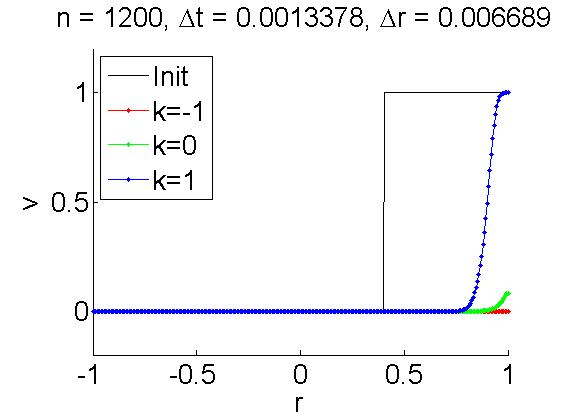}
\caption{The numerical solutions given by the Godunov scheme with a rarefaction for the particular  cases  $k=-1$, $k=0$ and $k=1$.}
\end{figure}

\begin{figure}
\centering
\includegraphics[width=5.5cm]{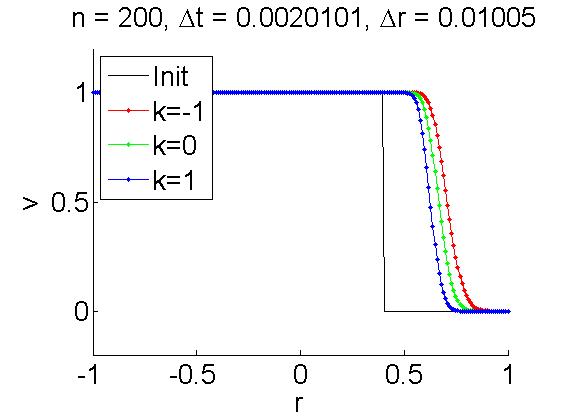} \includegraphics[width=5.5cm]{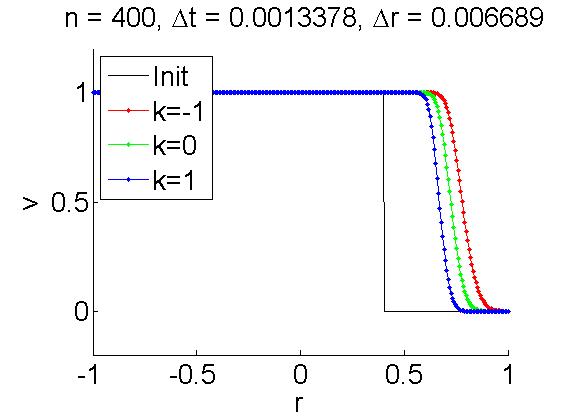} \\
\includegraphics[width=5.5cm]{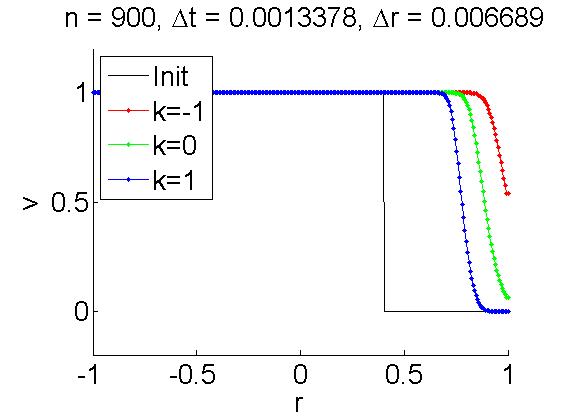} \includegraphics[width=5.5cm]{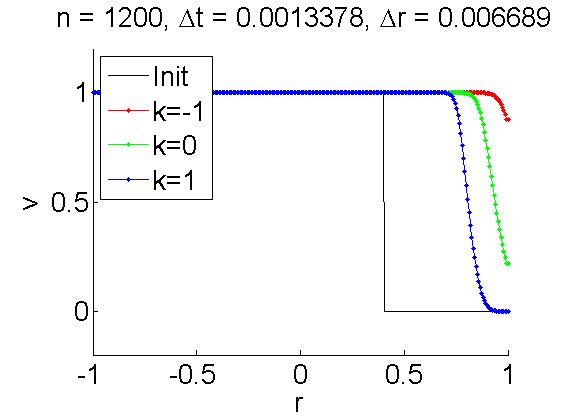}
\caption{The numerical solutions given by the Godunov scheme with a shock for the particular  cases  $k=-1$, $k=0$ and $k=1$.}
\end{figure}

\begin{figure}
\centering
\includegraphics[width=5.5cm]{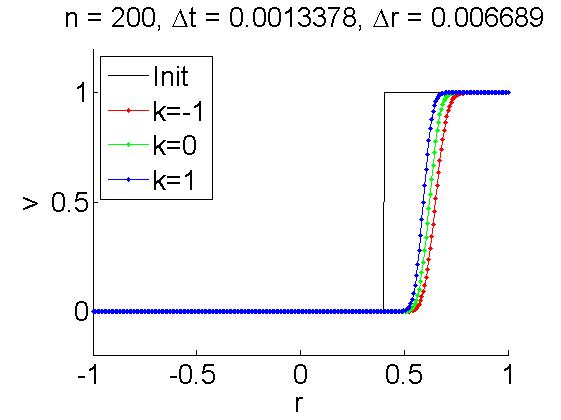} \includegraphics[width=5.5cm]{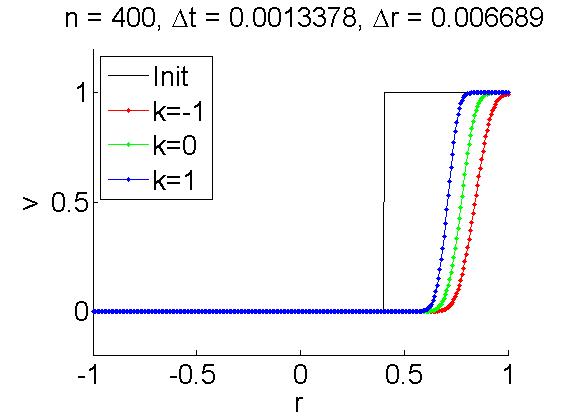} \\
\includegraphics[width=5.5cm]{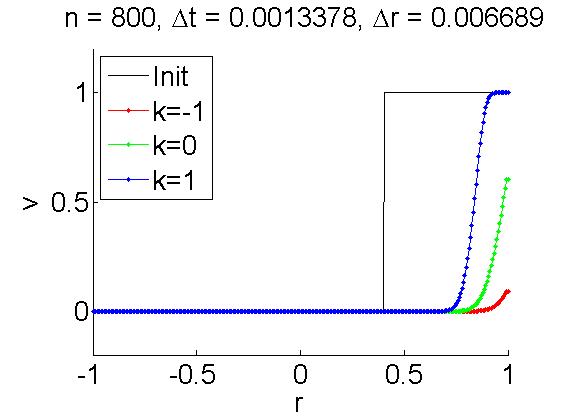} \includegraphics[width=5.5cm]{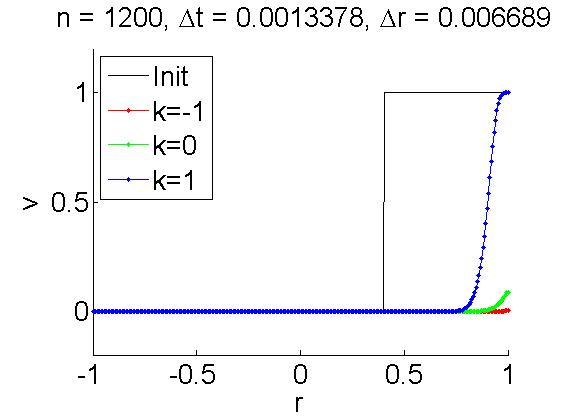}
\caption{The numerical solutions given by the second-order well-balanced Godunov scheme with a rarefaction for the particular  cases  $k=-1$, $k=0$ and $k=1$.}
\end{figure}

\begin{figure}
\centering
\includegraphics[width=5.5cm]{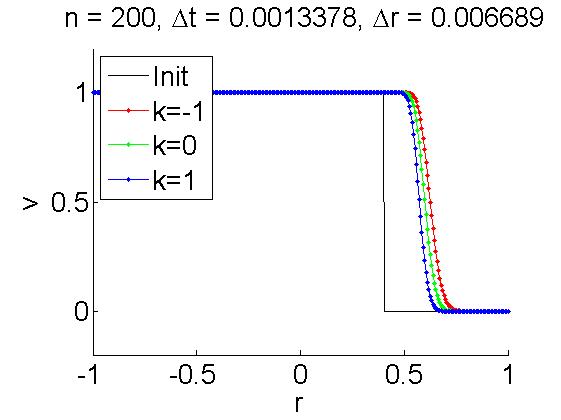} \includegraphics[width=5.5cm]{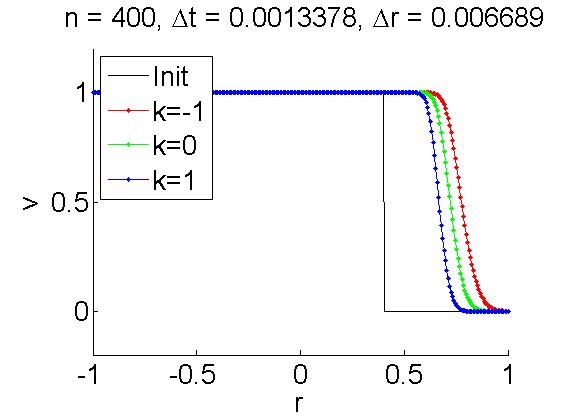} \\
\includegraphics[width=5.5cm]{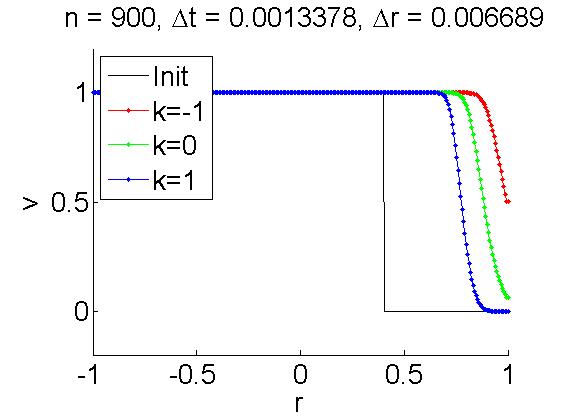} \includegraphics[width=5.5cm]{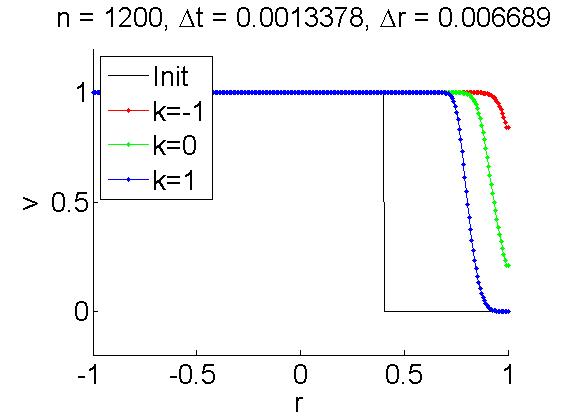}
\caption{The numerical solutions given by the  second-order well-balanced Godunov scheme with a shock for the particular  cases  $k=-1$, $k=0$ and $k=1$.}
\end{figure}

%=======================================================================================
\section{Concluding remarks}

In this paper, we have studied a nonlinear hyperbolic model which describes the propagation and interactions of shock waves on a Friedmann--Lema\^{i}tre--Robertson--Walker background spacetime.
We started from the relativistic Euler equations on a curved background and imposed a vanishing pressure in the expression of the energy--momentum tensor for perfect fluids. This led us to a {\sl geometric relativistic Burgers equation} (see \eqref{Euler4}) on the background spacetime under consideration. On a FLRW spacetime, the equation \eqref{Euler4} yields the model \eqref{lasteq} of interest in the present work.
The model involves a scale factor $a=a(t)$ which depends on the so-called 'cosmic time' and a constant coefficient $k$, which can be normalized to take the values $\pm1$ or $0$. The standard Burgers equation is recovered when $a(t)$ and $k$ are chosen to vanish.
We have then established various mathematical properties concerning the
hyperbolicity, genuine nonlinearity, shock waves, and rarefaction waves,
and we studied the class of spatially homogeneous solutions.

We have investigated shock wave solutions to our model for the three possible values of the coefficient $k$.

\begin{itemize}

\item We compared numerical solutions for  the cases $k=-1$, $k=0$ and $k=1$, and we found that the solution curve corresponding to $k=-1$ converges faster than the solution curve corresponding to $k=0$ and $k=1$ (Figures $1,2,3$ and $4$). This can be explained from the equation \eqref{lasteq} by observing that the characteristic speed
$\big(1-kr^2 \big)^{1/2}$ is increased by decreasing $k$.

\end{itemize}

Our analysis relies on a proposed numerical discretization scheme which applies to
discontinuous solutions and is based on the finite volume technique.

\begin{itemize}

\item Our scheme is  consistent with the conservative form of (the principal part of) our model and therefore correctly compute weak solutions containing shock waves.

\item Importantly, the proposed scheme is well-balanced, in the sense that it preserves (at the discrete level of approximation) all spatially homogeneous solutions.

\item Our numerical experiments illustrate the convergence, efficiency and robustness of the proposed
scheme on a FLRW background.

\end {itemize}

To conclude, we emphasize that the proposed methodology leading to a geometric relativistic balance law
may be used to derive other relativistic versions of Burgers equations on various classes of spacetimes. The advantages of such simplified nonlinear hyperbolic models is that they allow one to develop and test numerical methods for shock capturing and to reach definite conclusions concerning their convergence, efficiency, etc. Future work may include more singular backgrounds. Depending upon the particular background geometry, different techniques may be required in order to guarantee that certain classes of solutions of particular interest be preserved by the scheme, as we achieved it for time-dependent solutions.

%==================================================================================================================

\medskip

\

{\bf Acknowledgement.} The first author (T.C.) and the third author
(B.O.) were supported by the { \it Rectorate of Middle East Technical
University} (METU) through the grant "Project
BAP--08--11--2013--041". They were  also partially supported by the  {\it  Scientific and Technical Research Council of Turkey} (T\"{U}B\.{I}TAK)  through the grant "Ph.d. BIDEB 2214-A Scholarship Program".

%================================================================


\begin{thebibliography}{10}


\bibitem{ALO} \auth{P. Amorim, P.G. LeFloch, and B. Okutmustur},
Finite volume schemes on Lorentzian manifolds,
Comm. Math. Sc. 6 (2008), 1059--1086.


\bibitem{VG}  \auth{V. Guinot,}
{\sl Godunov--type schemes: an introduction for engineers,} Elsevier, 2003.

\bibitem{HEL} \auth{M. P. Hobson, G. P. Efstathiou, and  A. N. Lasenby},
{\sl General relativity. An introduction for physicists,}
Cambridge University Press, 2006.

\bibitem{SNK} \auth{S.N. Kruzkov},
First-order quasilinear equations in several independent variables,
Mat. Sbornik 81 (1970), 285--355; English trans. in Math. USSR Sb. 10
(1970), 217--243.

\bibitem{LO}  \auth{P.G. LeFloch and B. Okutmustur,}
Hyperbolic conservation laws on manifolds with limited regularity,
C.R. Math. Acad. Sc. Paris 346 (2008), 539--543.

\bibitem{LO2}  \auth{P.G. LeFloch and B. Okutmustur,}
Hyperbolic conservation laws on spacetimes. A finite
volume scheme based on differential forms,
Far East J. Math. Sci.  31 (2008), 49--83.

\bibitem{LNO}  \auth{P.G. LeFloch, W. Neves, and B. Okutmustur,}
Hyperbolic conservation laws on manifolds. Error estimate for finite volume schemes,
Acta Math. Sinica 25 (2009), 1041--1066.

\bibitem{LMO} \auth{P.G. LeFloch, H. Makhlof, and B. Okutmustur,}
Relativistic Burgers equations on a curved spacetime. Derivation and finite volume approximation,
SIAM J. Num. Anal. 50 (2012), 2136--2158.

\bibitem{Russo1} \auth{G. Russo,}
Central schemes for conservation laws with application to shallow water equations, S. Rionero, G. Romano Ed., Trends and Applications of Mathematics to Mechanics: STAMM 2002, Springer Verlag, Italy, 2005, pp.~225--246.

\bibitem{Russo2} \auth{G. Russo,}
High-order shock-capturing schemes for balance laws, in
“Numerical solutions of partial differential equations”, Adv. Courses Math. CRM Barcelona, Birkh\"auser, Basel, 2009, pp.~59--147.

\bibitem{L}  \auth{B. Van Leer,}
On the relation between the upwind-differencing schemes of Godunov,
Engquist-Osher and Roe, SIAM J. Sci. Stat. Comput. 5 (1984), 1--20.

\end{thebibliography}
\end{document}